\documentclass[11pt]{amsart}

\usepackage[english]{babel}
\usepackage{amsthm}

\makeatletter
\renewcommand{\tocsection}[3]{%
  \indentlabel{\@ifnotempty{#2}{\bfseries\ignorespaces#1 #2\quad}}\bfseries#3}
\renewcommand{\tocsubsection}[3]{%
  \indentlabel{\@ifnotempty{#2}{\ignorespaces#1 #2\quad}}#3}

\newcommand\@dotsep{4.5}
\def\@tocline#1#2#3#4#5#6#7{\relax
  \ifnum #1>\c@tocdepth 
  \else
    \par \addpenalty\@secpenalty\addvspace{#2}%
    \begingroup \hyphenpenalty\@M
    \@ifempty{#4}{%
      \@tempdima\csname r@tocindent\number#1\endcsname\relax
    }{%
      \@tempdima#4\relax
    }%
    \parindent\z@ \leftskip#3\relax \advance\leftskip\@tempdima\relax
    \rightskip\@pnumwidth plus1em \parfillskip-\@pnumwidth
    #5\leavevmode\hskip-\@tempdima{#6}\nobreak
    \leaders\hbox{$\m@th\mkern \@dotsep mu\hbox{.}\mkern \@dotsep mu$}\hfill
    \nobreak
    \hbox to\@pnumwidth{\@tocpagenum{\ifnum#1=1\bfseries\fi#7}}\par
    \nobreak
    \endgroup
  \fi}
\AtBeginDocument{%
\expandafter\renewcommand\csname r@tocindent0\endcsname{0pt}
}
\def\l@subsection{\@tocline{2}{0pt}{2.5pc}{5pc}{}}
\makeatother

\usepackage[numbers]{natbib}
\usepackage[T1]{fontenc}

\usepackage{url,xspace,smfthm}
\usepackage{amsmath}
\usepackage{amssymb}
\usepackage{mathrsfs}
\usepackage{stmaryrd}
\usepackage{yhmath}
\usepackage{dsfont}
\usepackage{mathtools}
\usepackage{xcolor}
\usepackage{natbib}
\usepackage{float}
\usepackage{enumitem}   

\usepackage{hyperref}
\hypersetup{colorlinks,linkcolor={blue},citecolor={blue},urlcolor={red}}

\definecolor{Red}{rgb}{1 0 0}
\definecolor{Green}{rgb}{0 1 0}
\definecolor{Blue}{rgb}{0 0 1}

\usepackage{geometry}
\usepackage{graphicx}
\graphicspath{{photos/}}

\newcommand{\application}[5]{\begin{array}{lrcl}
#1: & #2 & \longrightarrow & #3 \\
    & #4 & \longmapsto & #5 \end{array}}

\newcommand{\tr}[0]{\mathrm{tr}}

\newcommand{\Sp}[0]{\mathrm{Sp}}

\newcommand{\N}[0]{\mathbb{N}}

\newcommand{\R}[0]{\mathbb{R}}

\newcommand{\1}[0]{\mathbf{1}}
\newcommand{\Span}[0]{\operatorname{Span}}
\newcommand{\Hess}[0]{\operatorname{Hess}}
\newcommand{\rank}[0]{\operatorname{rank}}

\newcommand{\DD}{\Delta\!\!\!\!\Delta}

\makeatletter
\def\moverlay{\mathpalette\mov@rlay}
\def\mov@rlay#1#2{\leavevmode\vtop{%
   \base\Spaneskip\z@skip \\Spaneskiplimit-\maxdimen
   \ialign{\hfil$\m@th#1##$\hfil\cr#2\crcr}}}
\newcommand{\charfusion}[3][\mathord]{
    #1{\ifx#1\mathop\vphantom{#2}\fi
        \mathpalette\mov@rlay{#2\cr#3}
      }
    \ifx#1\mathop\expandafter\displaylimits\fi}
\makeatother

\newtheorem{theorem}{Theorem}[section]
\newtheorem{proposition}[theorem]{Proposition}
\newtheorem{corollary}[theorem]{Corollary}
\newtheorem{lemma}[theorem]{Lemma}
\newtheorem{definition}[theorem]{Definition}
\newtheorem{remark}[theorem]{Remark}
\newtheorem{example}[theorem]{Example}

\title{Isometric embedding of the $n$-point spaces into the space of spaces for $n \leq 4$}
\author{Benjamin Capdeville}

\begin{document}
\def\smfbyname{}

\maketitle

\begin{abstract}
In \cite{sturm2023space}, Sturm studied the space of all metric measure spaces up to isomorphism which he called \textit{The space of spaces}. He also introduced for $n \in \N^{*}$ the space of all $n$-points metric spaces. The aim of this article is to study if the embedding of this space in the space of spaces is isometric. Using results from \cite{maron2018probably} and \cite{maehara2013euclidean}, we prove that it is the case for $n\leq4$ and for Euclidean metric spaces.

\end{abstract}

\tableofcontents

\section{Introduction}

In his paper \textit{The space of spaces} \cite{sturm2023space}, Sturm introduces two distances for the space of metric spaces with $n\in \N^*$ points (up to isomorphism). The difference between these two distances is that the first one relies on transport maps whereas the second relies on transport plans. The goal of this article is to study if these two distances are the same.
In Section \ref{space of spaces}, we recall the framework of \textit{the space of spaces} and of \textit{the n-point spaces}, leading us to state our problem which turns out to be a quadratic optimization problem.
In Section \ref{proof}, following \cite{maron2018probably}, we find a sufficient condition for our problem, based on a convexity argument. This leads us to introduce the notion of metric spaces of negative type. In Section \ref{negative type}, following \cite{maehara2013euclidean}, we show that Euclidean spaces (Corollary \ref{eucl}) and metric spaces with $n\leq4$ points (Theorem \ref{four}) are of negative type, and that they indeed verify the identity we want to prove.

\section{Acknowledgements}

This article is part of my work during my internship at the University of Bonn, in Professor Sturm's team, as a part of my first year of master at ENS Paris-Saclay and at Paris-Saclay University. I am grateful to  Professor Sturm for welcoming and guiding me through this interniship.

\section{The space of spaces, the n-point spaces}\label{space of spaces}

In this part we recall the framework of \textit{the space of spaces} introduced by Sturm in \cite{sturm2023space}.

\subsection{The space of spaces}

\begin{definition}
A metric space measure space is $\mathbb{X} = (X,d,m)$ where $(X,d)$ is a complete separable metric space and $m$ is a Borel probability measure on $X$. For the sake of simplicity we suppose full support of the measure.
\end{definition}

\begin{definition}
Two metric measure spaces $\mathbb{X} = (X,d_X,m_X)$ and $\mathbb{Y} = (Y, d_Y,m_Y)$ are isomorphic when there exists a Borel measurable bijection $\phi : X \to Y$ with Borel measurable inverse such that $\phi_* m_X = m_Y$ and $d_X = \phi^* d_Y$.
\end{definition}

\begin{definition}
For $p \in [1, +\infty]$, and two metric measure spaces $\mathbb{X} = (X, d_x, m_X)$ and $\mathbb{Y} = (Y, d_Y, m_Y)$, we define their $L^p$-distortion distance :
$$\DD_p(\mathbb{X},\mathbb{Y}) = \left( \inf_{\pi \in Cpl(m_X, m_Y)} \int_{X \times Y} \int_{X \times Y} | d_X(x,x') - d_Y(y,y') |^p d\pi(x,y) d\pi(x',y') \right)^{\frac{1}{p}}.$$
\end{definition}

It is a pseudo-distance between metric measure spaces.
It is also known as the Gromov-Wasserstein distance \cite{memoli2011gw}.

$\DD_p$ is only a pseudo-distance, and the following proposition provides a characterization for two metric measure spaces whose $\DD_p$ distance is equal to 0.

\begin{proposition}
$\DD_p(\mathbb{X},\mathbb{Y}) = 0$ if and only if $\mathbb{X}$ and $\mathbb{Y}$ are isomorphic.
\end{proposition}

Being isomorphic is an equivalence relation so we consider metric measures spaces up to isomorphism.

\begin{definition}
The family of all isomorphism classes of metric measure spaces (with complete separable
metric and normalized volume, as usual) will be denoted by $\mathbb{X}_0$.
We will denote an equivalence class of mm-spaces by $\mathcal{X} = [X, d_X, m_X]$.

We also consider for $p \in [1, +\infty]$ mm-spaces with finite $L^p$ size :

For $\mathcal{X} = [X,d,m]$ a mm-space, let's denote $size_p(\mathcal{X}) = \left( \int_X \int_X d(x,y) dm(x) dm(y) \right)^{\frac{1}{p}}$
$$\mathbb{X}_p = \lbrace \mathcal{X} \in \mathbb{X}_0 : size_p(\mathcal{X}) < + \infty \rbrace.$$
\end{definition}

\begin{proposition}
$(\mathbb{X}_p, \DD_p)$ is a length metric space. It is separable but not complete.
\end{proposition}

\begin{proposition}
The metric completion of $(\mathbb{X}_p, \DD_p)$ is the space of all pseudo metric measure spaces up to isomorphism (still denoted the same).
\end{proposition}

\subsection{Gauged spaces}

The space of spaces, because of the constraints of distances, lacks of linearity. That's why Sturm introduced the notion of gauged spaces, in order to introduced similar notions than previously, but with a more linear structure.

\begin{definition}
A gauged measure space is $(X, f, m)$ where $X$ is a Polish space, $m$ is a Borel probability measure over $X$, and $f \in L^2_s(X \times X, m \otimes m)$, a symmetric $L^2$ function in $X \times X$, called a gauge.
\end{definition}

\begin{definition}[$L^2$-distortion distance]
Let $(X_1,f_1,m_1),(X_2,f_2,m_2)$ be two gauged measure spaces, there $L^2$-distortion distance is defined by :
\begin{equation*}
\begin{aligned}
\DD&((X_1,f_1,m_1),(X_2,f_2,m_2)) = \\ 
&\inf \Bigg\lbrace \left( \int_{X_1 \times X_2} \int_{X_1 \times X_2} | f_1(x_1,y_1) - f_2(x_2,y_2) |^2 d\bar{m}(x_1,x_2) d\bar{m}(y_1,y_2) \right)^{\frac{1}{2}} : \bar{m} \in Cpl(m_1,m_2) \Bigg\rbrace.
\end{aligned}
\end{equation*}
\end{definition}

\begin{definition}
Let $(X_1,f_1,m_1),(X_2,f_2,m_2)$ be two gauged measure spaces, there are called homomorphic if $ \DD((X_1,f_1,m_1),(X_2,f_2,m_2)) = 0$. It is an equivalence relation.
\end{definition}

\begin{definition}
The space of equivalence classes of homomorphic gauged spaces will be denoted $\mathbb{Y}$. It is complete with respect to the $L^2$-distortion distance.
\end{definition}

\subsection{The n-point spaces as an injection into the space of spaces}

We focus now on spaces with a finite number of points, and we study how they combine with the framework of metric measure spaces introduced in the previous sections.

Let $n \in \N^*$.

\begin{definition}
Let $(\lbrace x_1, ..., x_n \rbrace, f)$ be a finite gauged space. Its gauged matrix is $\tilde{f} \in \mathcal{S}_n(\R)$ defined by :
$$\forall i,j \in \llbracket 1,n \rrbracket, \tilde{f}_{ij} = f(x_i, x_j). $$
Same definition applies for metric spaces and distance matrices.
\end{definition}

We state the definition of isomorphic gauged (resp. metric) 
spaces in the context of spaces with a finite number of points.

\begin{definition}
Two $n$-point gauged (resp. metric) spaces $X = (\lbrace x_1, ..., x_n \rbrace, f)$ and $Y = (\lbrace y_1, ..., y_n \rbrace, g)$ are isomorphic if there exists $\phi : X \to Y$ such that for all $i,j \in \llbracket 1,n \rrbracket$, $f(x_i, x_j) = g(\phi(x_i), \phi(x_j)).$
\end{definition}

Let's emphasize that two isomorphic spaces does not necessarily have the same gauged (reps. distance) matrix. Indeed, when the order of the points is fixed, an isomorphism can be identified as a permutation $\sigma \in \mathfrak{S}_n$ such that $\forall i,j \in \llbracket 1,n \rrbracket, f(x_i, x_j) = g(\phi(x_i) \phi(x_j)) = g(y_{\sigma_i}, y_{\sigma_j})$. In this case we note $\tilde{f} = \sigma^* \tilde{g}$. Thus, two spaces are isomorphic if and only if there exists a $\sigma \in \mathfrak{S}_n$ such that for all $i,j \in \llbracket 1,n \rrbracket, \tilde{f}_{ij} = \tilde{g}_{\sigma_i \sigma_j}$.

These considerations leads us to introduce the following definitions when considering gauged (resp. distance) spaces with matrices :

\begin{definition}
We denote $M^{(n)}$ the space of $(n \times n)$ symmetric matrices with null diagonal, with the $l_2$-norm :

$$\forall f \in M^{(n)}, \|f\|^2 = \frac{1}{n^2} \sum_{i,j=1}^n f^2_{ij}.$$

We denote $\mathbb{M}^{(n)}$ the space $M^{(n)} / \mathfrak{S}_n$. It is equipped with the quotient distance induced by the $l_2$-norm :

\begin{align*}
    \forall f,g \in \mathbb{M}^{(n)}, d_{\mathbb{M}^{(n)}}(f,g) &= \min_{\sigma \in \mathfrak{S}_n} \|f - \sigma^*g \| \\
    &= \min_{\sigma \in \mathfrak{S}_n} \Big( \frac{1}{n^2}\sum_{i,j=1}^n | f_{ij} - g_{\sigma_i \sigma_j} |^2 \Big)^\frac{1}{2}.
\end{align*}

\end{definition}

\begin{definition}
Similarly we define $M^{(n)}_{\leq}$ the subspace of $M^{(n)}$ such that \textit{the triangle inequality is verified :} $$D \in  M^{(n)}_{\leq} \quad \text{iff} \quad  D \in M^{(n)} \quad \text{and} \quad \forall i,j,k \in \llbracket 1,n \rrbracket, D_{ij} \leq D_{ik} + D_{kj}.$$

$\mathfrak{S}_n$ still acts isometrically on $M^{(n)}_{\leq}$ so as previously we define $\mathbb{M}^{(n)}_{\leq} = M^{(n)}_{\leq} / \mathfrak{S}_n$. 
\end{definition}

Hence as announced, $\mathbb{M}^{(n)}$ represents gauged spaces with $n$ points and $\mathbb{M}^{(n)}_{\leq}$ represents pseudo-metric spaces with $n$ points. For the sake of simplicity, when speaking of finite metric spaces we will always mean pseudo-metric spaces. Note that all the reasoning that follows still apply.

Sturm provides a study of the geometric properties of these spaces, but here we will focus on metric considerations.

These spaces naturally injects into the space of spaces :

$$\application{\Phi}{\mathbb{M}^{(n)}}{\mathbb{Y}}{f}{\llbracket \llbracket 1,n \rrbracket, f, \frac{1}{n} \sum_{i=1}^n \delta_i \rrbracket }$$

Note that if $D \in \mathbb{M}^{(n)}_{\leq}$ then $\Phi(D) \in \mathbb{X}$.

\begin{definition}[$n$-point spaces]
We define $\mathbb{Y}^{(n)} := \Phi(\mathbb{M}^{(n)})$ and $\mathbb{X}^{(n)} := \Phi(\mathbb{M}^{(n)}_{\leq})$.
\end{definition}

\subsection{Statement of the problem}

The question is to prove whether these injections are isometric embeddings.

Some basic computations and considerations about bistochastic matrices gives the following :

\begin{proposition}\label{ineq}
For all $n \in \N^{*}, \Phi$ is $1$-lipschitz, \textit{i.e.}
$$\forall f,g \in \mathbb{M}^{(n)}, \DD( \Phi(f), \Phi(g)) \leq d_{\mathbb{M}^{(n)}}(f,g).$$
\end{proposition}

\begin{proof}

Let $f,g \in \mathbb{M}^{(n)}$.
Here, transport plans between $f$ and $g$ can be seen as bistochastic matrices, \textit{i.e.} matrices where the sum of the terms in each line and in each column is equal to one, we denote this set $\mathcal{B}_n$.

Thus we have :
\begin{align*}
    \DD(\Phi(f), \Phi(g))^2 &=  \min_{P \in \mathcal{B}_n} \frac{1}{n^2} \sum_{i,j,k,l=1}^{n} |f_{ij}-g_{kl}|^2 p_{ik}p_{jl} \\
    &= \|f\|^2 + \|g\|^2 - \frac{2}{n^2} \max_{P \in \mathcal{B}_n} \sum_{i,j,k,l=1}^{n} f_{ij}g_{kl} p_{ik}p_{jl}.
\end{align*}

On the other hand we have :
\begin{align*}
    d_{\mathbb{M}^{(n)}}(f,g)^2 = \|f\|^2 + \|g\|^2 - \frac{2}{n^2} \max_{\sigma \in \mathfrak{S}_n} \sum_{i,j=1}^n f_{ij}g_{\sigma_i \sigma_j}.
\end{align*}

Remark that permutation matrices (\textit{i.e.} matrices  $(P_{\sigma})_{ij} = \1_{\sigma_i=j}$ where $\sigma \in \mathfrak{S}_n$) are bistochastic matrices, and that for $\sigma \in \mathfrak{S}_n$ and $i,j,k,l \in \llbracket 1,n \rrbracket, f_{ij} g_{kl} p_{{\sigma}_{ik}} p_{{\sigma}_{jl}} = f_{ij}g_{\sigma_i \sigma_j} \1_{\sigma_i = k} \1_{\sigma_j = l}$. It yields :
$$\max_{\sigma \in \mathfrak{S}_n} \sum_{i,j=1}^n f_{ij}g_{\sigma_i \sigma_j} \leq \max_{P \in \mathcal{B}_n} \sum_{i,j,k,l=1}^{n} f_{ij}g_{kl} p_{ik}p_{jl} $$ 
and thus the desired inequality.

\end{proof}

To prove that we have an isometric embedding, it remains to prove that it is indeed an embedding (this is the case according to Theorem \ref{embed_sturm}), and that it is an isometry, i.e. that the reverse inequality of Proposition \ref{ineq} holds. According to the last proof, this is a quadratic optimization problem, where permutation and bistochastic matrices are involved.

The difference between these two distances is very common in Optimal Transport theory. As in the Monge problem, $d_{\mathbb{M}^{(n)}}$ is based on transport maps, which always exist here only because we consider uniformly weighted spaces with the same number of points, whereas $\DD$ adopts the Kantorovitch point of view, relying on transport plans.

The link between permutation matrices and bistochastic matrices is actually deeper as stated by this well-known theorem :

\begin{theorem}[Birkhoff - von Neumann]\label{BVN}
$\mathcal{B}_n$ is a compact convex polytope whose extreme points are exactly permutation matrices.
\end{theorem}

This theorem gives a geometric interpretation of our quadratic optimization problem : does some function take extremal value at an extrem point of our set ?

In the context of metric spaces, Sturm proved that the injection of $\mathbb{X}^{(n)}$ in $\bar{\mathbb{X}}$ is an embedding.

\begin{theorem}\label{embed_sturm}
The injection $\Phi : \mathbb{M_{\leq}}^{(n)} \to \bar{\mathbb{X}}$ is an embedding, \textit{i.e.} an homeomorphism onto its image.
\end{theorem}

\begin{proof}
See discussion in \cite{sturm2023space}, Remark 5.26.
\end{proof}

To sum up everything, the problem is the following : prove or disprove 

\begin{equation}\label{problem}
    \max_{\sigma \in \mathfrak{S}_n} \sum_{i,j=1}^n f_{ij}g_{\sigma_i \sigma_j} = \max_{P \in \mathcal{B}_n} \sum_{i,j,k,l=1}^{n} f_{ij}g_{kl} p_{ik}p_{jl}
\end{equation}

whether when $f,g$ are $(n \times n)$ symmetric matrices with null diagonal, or when $f,g$ are distance matrices of some $n$-point metric spaces.

\subsection{Counterexample for gauged spaces}
We prove that equality (\ref{problem}) does not hold in general for gauged spaces.

Let's denote $f$ the $n \times n$ matrix where $f_{ij} = \1_{i \neq j}$, and $g = - f$. We have $f,g \in M^{(n)}$. By notation abuse we will still write their image in the quotient space $f,g \in \mathbb{M}^{(n)}$.

Remark that for all $\sigma \in \mathfrak{S}_n$, we have $\sigma_* g= g$ so :
\begin{align*}
    d_{\mathbb{M}^{(n)}}(f,g) &= \min_{\sigma \in \mathfrak{S}_n} \|f - \sigma^*g \| \\
    &= \| f-g \| \\
    &= \Big( \frac{1}{n^2} \sum_{i,j=1}^n 4 \times \1_{i \neq j} \Big)^\frac{1}{2}\\
    &= 2\sqrt{1 - \frac{1}{n}}.
\end{align*}

We denote $M$ the $n \times n$ matrix where all elements are equal to $\frac{1}{n}$, thus $M \in \mathcal{B}_n$.

\begin{align*}
    \DD(\Phi(f), \Phi(g)) &\leq \Big( \frac{1}{n^2} \sum_{i,j,k,l=1}^n |f_{ij}-g_{kl}|^2 M_{ik}M_{jl} \Big)^{\frac{1}{2}} \\
    &\leq \Big( \frac{1}{n^4} \sum_{i,j,k,l=1}^n (\1_{i\neq j} + \1_{k \neq l} )^2 \Big)^{\frac{1}{2}} \\
    &\leq \Big( \frac{1}{n^4} (2n^2 (n^2-n) + 2 (n^2-n)^2) \Big)^{\frac{1}{2}} \\
    &\leq \left( 4-\frac{6}{n} + \frac{1}{n^2} \right)^{\frac{1}{2}}
\end{align*}

Then, 
\begin{align*}
    \left( 4-\frac{6}{n} + \frac{1}{n^2} \right)^{\frac{1}{2}} < 2\sqrt{1 - \frac{1}{n}} &\iff 4n^2 - 6n + 1 < 4n^2 - 4n \\
    &\iff 1 < 2n \\
    &\iff 1 < n.
\end{align*}

This proves that $\DD(\Phi(f), \Phi(g)) < d_{\mathbb{M}^{(n)}}(f,g)$ and so that $\Phi$ is not an isometry.

Thus, the question is to determine whether the restrictions induced by the metric setting are enough to prove equality (\ref{problem}).

\section{Proof for spaces of negative type}\label{proof}

In this section we prove that (\ref{problem}) holds when considering metric spaces of negative type.

We fix $D^X, D^Y \in \mathbb{M}_{\leq}^{(n)}$.

Let's introduce $$\application{h}{\mathcal{B}_n}{\R}{P}{\sum_{i,j,k,l=1}^n D^X_{ij}D^Y_{kl} p_{ik} p_{jl}}.$$

Remark that for $P \in \mathcal{B}_n$, we have 
\begin{equation}\label{ecr_tr}
    h(P) = \tr(P^T D^X P (D^Y)^T).
\end{equation}

According to Theorem \ref{BVN}, equality (\ref{problem}) is equivalent to :
\begin{equation}\label{eq2}
    \max_{\mathcal{B}_n} h = \max_{\operatorname{ext}(\mathcal{B}_n)} h.
\end{equation}

The proof is inspired by \cite{maron2018probably} and goes as follows : we show that we have equality (\ref{eq2}) when $h$ is convex, so first we find a sufficient condition for h to be convex. Then, using property of Kronecker product this gives us a condition on the tensor product of the distance matrices of each space. Finally, we prove that this is verified if each space is of negative type.

\subsection{A convexity argument}

The idea of the proof is to find a sufficient condition for $h$ to be convex. Thus by the following corrolary of the Krein-Milman theorem, we have equality (\ref{eq2}).

\begin{lemma}\label{key}
    Let $K \subset \R^k$ be a convex compact nonemptset, and $f : K \to \R$ be a convex function. Then $\max_K f = \max_{ext(K)} f$.
\end{lemma}

\begin{proof}
    Let $y \in K$ such that $f(y) = \max_K f$. By Krein-Milman theorem there exists ${n \in \N}, \alpha_1, ..., \alpha_n \in \R_+, x_1, ..., x_n \in Ext(K)$ such that $\sum_{i=1}^n \alpha_i = 1$ and $y = \sum_{i=1}^n \alpha_i x_i$.
    $f$ is convex so $\max_K f = f(y) \leq \sum_{i=1}^n \alpha_i f(x_i) \leq \max_K f$. So for $1 \leq i \leq n, f(x_i) = \max_K f$ and $\max_K f = \max_{ext(K)} f$.
\end{proof}

For smooth functions, a slight modification of the second order characterization of convex function will give us a sufficient condition in terms of the hessian.

\begin{proposition}
    Let $f : \R^k \to \R$ be a $\mathcal{C}^2$ function, let $K \subset \R^k$ be a convex compact set. If $\forall x \in K, \forall v \in \Span(K), v^T (\Hess f)(x) v \geq 0$ then $f\vert_K$ is convex.
\end{proposition}

\begin{proof}
    $f\vert_K$ is convex if and only if it is convex along any line in $K$. For $a,b \in K$, we set for $\lambda \in [0,1], g_{a,b}(\lambda) = f(\lambda a + (1-\lambda)b)$. $f\vert_K$ is convex iff $g_{a,b}$ is convex for all $a,b \in K$. $g_{a,b}$ is $\mathcal{C}^2$ and it is convex if for all $\lambda \in [0,1], g_{a,b}''(\lambda) \geq 0$, i.e. $(a-b)^T(Hess f)(b + \lambda (a-b))(a-b) \geq 0$. This is true for every $a,b \in K, \lambda \in [0,1]$ according to our hypothesis. So $f\vert_K$ is convex.
\end{proof}

Combining both proposition gives :

\begin{corollary}\label{cor}
Let $f : \R^k \to \R$ be a $\mathcal{C}^2$ function, let $K \subset \R^k$ be convex compact. 

If $\forall x \in K, \forall v \in \Span(K), v^T (\Hess f)(x) v \geq 0$ then $\max_K f$ is attained on an extreme point of $K$.
\end{corollary}

\subsection{Kronecker product and vectorization}
We recall a few properties of Kronecker product and vectorization. Proofs rely on computations that will not be explicited here. In the context of our problem it allows us to split the condition obtained in Corollary \ref{cor} into a condition in each space.

\begin{definition}
    Let $A$ be a $m \times n$ matrix and $B$ a $p \times q$ matrix. Its Kronecker product is the $pm \times qn$ matrix defined by $A \otimes B =
    \begin{pmatrix}
        a_{11} B &  ... & a_{1m}B \\
        \vdots & \ddots & \vdots \\
        a_{n1} B & \cdots & a_{nm} B
    \end{pmatrix}$.
    It is bilinear and associative.
\end{definition}

We will use this property later in our proof :

\begin{proposition}
    If A, B are square matrices of size $n$ and $m$, with eigenvalues $\lambda_1, ..., \lambda_n$ and $\mu_1, ..., \mu_m$ then the eigenvalues of $A \otimes B$ are $\lambda_i \mu_j$ for $1 \leq i \leq n, 1 \leq j \leq n$.
\end{proposition}

\begin{definition}
    Let $A$ be a $ m \times n$ matrix. Its vectorization is denoted $[A]$ and is the $mn \times 1$ matrix obtained by stacking its column on top of one another :
    $$[A] = (a_{11}, ..., a_{m1}, a_{21}, ..., a_{n1}, ..., a_{nm})^T.$$
\end{definition}

\begin{proposition}\label{base}
    Let $a \in \R^n, b \in \R^m$. Then, $[ab^T] = b \otimes a$.
\end{proposition}

\begin{proposition}
    Let $A,B,C,D$ be matrices such that the following operations make sense, we have :
    $$\tr(A^TBCD^T) = [A]^T (D \otimes B) [C].$$
\end{proposition}

Indeed with (\ref{ecr_tr}), we obtain :
$$\forall P \in \mathcal{B}_n, h(P) = \tr(P^TD^XP(D^Y)^T) = [P]^T (D^Y \otimes D^X) [P].$$

Thus, we have for all $P \in \Span(\mathcal{B}_n), (\Hess h)(P) = 2 D^Y \otimes D^X$. Consequently, according to Corollary \ref{cor}, if for all $x \in \Span(\mathcal{B}_n), [x]^T (D^Y \otimes D^X) [x] \geq 0$, i.e. if for all $x \in [\Span(\mathcal{B}_n)], x^T (D^Y \otimes D^X) x \geq 0$,  then (\ref{problem}) is true.

\subsection{Conditionally negative semi-definite matrices}

This condition can be split into a condition for each space. It relies on a decomposition of $[\Span(\mathcal{B}_n)]$ as a tensor product of two identical vector spaces. We need first to introduce \textit{conditionally negative semi-definite} matrices.

\begin{definition}
$A \in S_n(\R)$ is \textit{conditionally negative semi-definite} with respect to a subspace $V \subset \R^n$ if $$\forall x \in V, v^TAv \leq 0.$$
With $d = \dim(V)$ and $F = (f_1, ..., f_d)$ an orthonormal basis of $V$,
\begin{align*}
    \text{A is conditionally negative semi-definite with respect to V} &\iff \forall x \in \R^d, x^TF^TAFx \leq 0 \\
    &\iff {\Sp(F^TAF) \subset \R_-}.
\end{align*}
Similarly we say that $A \in S_n(\R)$ is \textit{conditionally positive semi-definite} with respect to $V$ if  ${\Sp(F^TAF) \subset \R_+}$.
\end{definition}

\begin{proposition}\label{tensor}
If $A$ (resp. $B$) are conditionally negative semi-definite with respect to $V$ and (resp. $W$), then $B \otimes A$ is conditionally positive semi-definite with respect to $V \otimes W$.
\end{proposition}

\begin{proof}
Let $F,G$ be orthonormal basis of $V$ and $W$, then $F \otimes G$ is an orthonormal basis of $V \otimes W$.
We have : $(F \otimes G)^T(A \otimes B) (F \otimes G) = (F^TAF) \otimes (G^TBG)$.

Then, $\Sp((F \otimes G)^T(A \otimes B) (F \otimes G)) = \lbrace \lambda \mu, \lambda \in \Sp(F^TAF), \mu \in \Sp(G^TBG) \rbrace$. 

$A,B$ are conditionally negative semi-definite with respect to $V$ and $W$ so $\Sp(F^TAF) \subset \R_-$ and $\Sp(G^TBG) \subset \R_-$. Thus, $\Sp((F \otimes G)^T(A \otimes B) (F \otimes G)) \subset \R_+$.
\end{proof}

\subsection{Negative type metric spaces}

Let's apply these results in the setting of our problem, thanks to the following lemma :

\begin{lemma}\label{decBn}
If $F = (f_1, ..., f_{n-1})$ is an orthonormal basis of $\1^{\perp}$ then $F \otimes F$ is an orthonormal basis of $[\Span(\mathcal{B}_n)]$.
\end{lemma}

\begin{proof}

$\mathcal{B}_n = \lbrace X \in \mathcal{M}_n(\R), X \1 = \1, X^T \1 = \1 \rbrace$ so $\Span(\mathcal{B}_n) = \lbrace X \in \mathcal{M}_n(\R), X \1 = 0, X^T \1 = 0 \rbrace$. It is a vector space of dimension $(n-1)^2$. Take $i,j \in \llbracket 1,n-1 \rrbracket$, we have $f_j f_i^T \in \Span(\mathcal{B}_n)$ so according to Proposition \ref{base}, we have $f_i \otimes f_j \in [ \Span(\mathcal{B}_n) ]$ which is also a vector space of dimension $(n-1)^2$. The columns of $F \otimes F$ are exactly $f_i \otimes f_j$ for $i,j \in \llbracket 1,n-1 \rrbracket$ and

\begin{align*}
    (F\otimes F)^T(F\otimes F) &= (F^TF) \otimes (F^TF) \\
    &= I_{n-1} \otimes I_{n-1} \\
    &= I_{(n-1)^2}
\end{align*}
Thus, ${(F\otimes F)^T(F\otimes F) \in O_{(n-1)^2}(\R)}$. This proves that the columns of $F \otimes F$ form an orthonormal basis of $[\Span(\mathcal{B}_n)]$.
\end{proof}

This leads to the following definition and a corollary :

\begin{definition}\label{neg_type}
We say that $D^X \in \mathbb{M}_{\leq}^{(n)}$ is of negative type if it is conditionally negative semi-definite with respect to  $\1^{\perp}$. It is equivalent to :
$$\forall v \in \R^n, \sum_{i=1}^n v_i = 0 \implies v^T D^X v \leq 0.$$
\end{definition}

\begin{corollary}
If $D^X,D^Y \in \mathbb{M}_{\leq}^{(n)}$ are of negative type then equality (\ref{eq2}) is true.
\end{corollary}

\begin{proof}
If $D^X$ and $D^Y$ are of negative type then thanks to Propostion \ref{tensor} and Lemma \ref{decBn}, $D^Y \otimes D^X$ is conditionally positive semi-definite with respect to $\1^{\perp} \otimes \1^{\perp} = [\Span(\mathcal{B}_n)]$. According to Lemma \ref{key} we have $\max_{\mathcal{B}_n} h = \max_{\operatorname{ext}(\mathcal{B}_n)} h$.
\end{proof}

We proved that (\ref{problem}) is true at least for metric spaces of negative type. The next section is devoted to the study of these spaces.

\section{Properties of metric spaces of negative type}\label{negative type}

In this section we present results from \cite{maehara2013euclidean} about negative type metric spaces.

First Gram matrices are introduced to determine when a metric space can be embedded into an Euclidean space. From that we can prove that a metric space is of negative type if and only if some of its power transform is Euclidean. Then, we show that being Euclidean is preserved by power transform.  With all these results, we have that metric spaces with less than 4 points and Euclidean metric spaces are of negative type.

\subsection{Embedding of metric spaces into an Euclidean space}

We introduce Gram matrices in order to characterize Euclidean metric spaces.

Let's fix $(X,d)$ a metric space with $m \in \N^*$ points, $X = \lbrace x_1, ..., x_m \rbrace$.

\begin{definition}
$(X,d)$ is said to be Euclidean if there exists an embedding of $(X,d)$ into $\R^n$ for some $n \in \N$.
\end{definition}

\begin{definition}
The Gram matrix of $(X,d)$ is the $(m-1) \times (m-1)$ matrix where the $(i,j)$-element  is $\frac{1}{2}(d(x_i, x_m)^2 + d(x_j, x_m)^2 - d(x_i, x_j)^2)$. Let's denote it $\mathcal{G}(X,d)$.
\end{definition}

\begin{remark}\label{corres}
    We can recover the distances between every point from the Gram matrice of a metric space : for $i,j \in \llbracket 1,n \rrbracket$, $d^2(x_m,x_i) = \mathcal{G}(X,d)_{i,i}$ and \mbox{$d^2(x_i,x_j) = \mathcal{G}(X,d)_{i,i} + \mathcal{G}(X,d)_{j,j} - 2 \mathcal{G}(X,d)_{i,j} $}.
\end{remark}

\begin{theorem}\label{embed0}
$(X,d)$ is embeddable in $\R^n$ if and only if $\mathcal{G}(X,d)$ is positive semidefinite with rank less than or equal to $n$.
\end{theorem}

\begin{proof}

\begin{itemize}
    \item If $(X,d)$ is Euclidean, there exists $n \in \N$ and $y_1, ..., y_m \in \R^n$ that verifies for ${i,j \in \llbracket 1,m \rrbracket}, d(x_i,x_j) = \|y_i - y_j\|$.
    
    For $i,j \in \llbracket 1,m-1 \rrbracket$, by the polarization identity we have :
    \begin{align*}
        \mathcal{G}(X,d)_{i,j} &= \frac{1}{2}( \|y_i-y_m\|^2 + \|y_j-y_m\|^2 - \|y_i-y_j\|^2) \\
        &= \langle y_i-y_m, y_j-y_m \rangle.
    \end{align*}
    We denote $Z$ the $(m-1) \times (m-1)$ matrix such that the $i$-th column is the vector $y_i-y_m$. Then we have $\mathcal{G}(X,d) = Z^T Z$.
    
    Let $v \in \R^{m-1}$, 
    \begin{align*}
        v^T \mathcal{G}(X,d) v &= (Zv)^T Zv \\
                              &= \|Zv\|^2 \\
                              &\geq 0. 
    \end{align*}
    
    So $\mathcal{G}(X,d)$ is positive semidefinite.
    Moreover, $(X,d)$ is embedabble into $\R^n$ so :
    
    \begin{align*}
        \rank(\mathcal{G}(X,d)) &= \rank(Z) \\
                                &= \dim( \Span( \lbrace y_i - y_m : i \in \llbracket 1,m-1 \rrbracket \rbrace)) \\
                                &\leq n.
    \end{align*}
    
    \item If $\mathcal{G}(X,d)$ is positive semidefinite with rank at most $n$, by Cholesky decomposition there exists a lower triangular matrix such that $\mathcal{G}(X,d) = L^T L$. Then $\rank(L) = \rank(\mathcal{G}(X,d)) \leq n$. Thus if we denote $L = (y_1, ..., y_{m-1})$ with $y_1, ..., y_{m-1} \in \R^{m-1}$, we have $\dim(\Span(\lbrace y_1, ..., y_{m-1} \rbrace)) \leq n$. 
    
    So, for $i,j \in \llbracket 1,m-1 \rrbracket$, $\mathcal{G}(X,d)_{ij} = (L^T L)_{ij} = y_i^T y_j = \langle y_i, y_j \rangle$.
    
    Then, 
    \begin{align*}
        \|y_i - y_j \|^2 &= \|y_i\|^2 - 2 \langle y_i, y_j \rangle + \|y_j\|^2 \\
        &= d(x_i, x_m)^2 - (d(x_i, x_m)^2 + d(x_j, x_m)^2 - d(x_i, x_j)^2) + d(x_j, x_m)^2 \\
        &= d(x_i, x_j)^2.
    \end{align*}
    Hence, the application $x_i \mapsto y_i$ is an isometry and we have an embdedding of $X$ into $\R^m$.

\end{itemize}

\end{proof}

The immediate corollary is less precise but it will be sufficient for our purpose.

\begin{corollary}\label{embed}
    $(X,d)$ is Euclidean if and only if $\mathcal{G}(X,d)$ is positive semi-definite.
\end{corollary}

\subsection{Power transform of a metric space}
Here we study what happens if we put all the distances in a metric space to the same power $c$. When $0<c \leq 1$ , the triangle inequality is preserved (i.e. we still have a metric space) and we will see later that it is as well the property of being Euclidean is also preserved (Theorem \ref{power_euclidean}).

\begin{lemma}\label{ineq_power}
    If $p + q \geq r$ are positive and $0 < c < 1$, then $p^c + q^c > r^c$.
\end{lemma}

\begin{proof}
    Let $f(x) = x^c$, it is strictly concave because $0<c<1$. By slope inequality we have $\frac{f(p) - f(0)}{p} > \frac{f(p+q) - f(q)}{p}$. $f(0) = 0$, $f$ is increasing and $p+q \geq r$ so $f(p) + f(q) > f(p+q) \geq f(r)$ so $p^c + q^c > r^c$.
\end{proof}

This lemma proves that the power transform of a metric space is a metric space.

\begin{corollary}
    If $(X,d)$ is a metric space and $0<c \leq 1$, then $(X,d^c)$ is a metric space.
\end{corollary}

\subsection{Link between Euclidean spaces and metric spaces of negative type}

Let $(X,d)$ be a finite metric space with $m \in \N^{*}$ points, $X = \lbrace x_1, ..., x_m \rbrace$.

The following theorem shows a link between Euclidean spaces and negative type spaces. This equivalence allows us to work in Euclidean spaces, which are more concrete than negative type spaces. It results from simple computations.

\begin{theorem}
$(X,d^{\frac{1}{2}})$ is Euclidean if and only if $(X,d)$ is of negative type.
\end{theorem}

\begin{proof}

\begin{itemize}
    \item If $(X,d^{\frac{1}{2}})$ is of Euclidean there exists $n \in \N$ and $y_1, ..., y_m \in \R^n$ such that \newline for  $i,j \in \llbracket 1,m \rrbracket, d^{\frac{1}{2}}(x_i,x_j) = \|y_i - y_j \|_2$. Let's denote $D$ the distance matrix of $(X,d)$ and let $\eta \in \1^{\perp}$ :
\begin{align*}
        \eta^T D \eta &= \sum_{i,j=1}^m \eta_i \eta_j ||y_i - y_j||_2^2 \\
                      &= \underbrace{\left( \sum_{i=1}^m \eta_i \right)}_{= 0} \left( \sum_{j=1}^m \eta_j ||y_j||_2^2 \right) + \underbrace{\left( \sum_{j=1}^m \eta_j \right)}_{= 0} \left( \sum_{i=1}^m \eta_i ||y_i||_2^2 \right) - 2 \sum_{i,j=1}^m \eta_i \eta_j < y_i, y_j > \\
                      &= -2 \left\lVert \sum_{i=1}^m \eta_i y_i \right\rVert^2_2 \\
                      &\leq 0.
\end{align*}
So $(X,d)$ is of negative type.
    \item Suppose that $(X,d)$ is of negative type. Let $v \in \R^{m-1}$. Let $\tilde{v} = \left( \tilde{v}_1, ..., \tilde{v}_{m-1}, \tilde{v}_m\right) \in \R^{m} $ such that \mbox{for $1\leq i \leq m-1, \tilde{v}_i = v_i$} and $\tilde{v}_{m} = - \sum_{i=1}^{m-1} v_i$, so we have $\tilde{v} \in \1^{\perp}$. According to {Remark \ref{corres}} we have :
    \begin{align*}
        \sum_{i,j=1}^{m} \tilde{v}_i \tilde{v}_j d(x_i,x_j) &= \sum_{j=1}^{m-1} \tilde{v}_m v_j d(x_m, x_j) + \sum_{i=1}^{m-1} v_i \tilde{v}_m d(x_i, x_m) + \sum_{i,j=1}^{m-1} v_i v_j d(x_i, x_j) \\
        &= 2 \tilde{v}_m \sum_{j=1}^{m-1} v_j \mathcal{G}(X,d^{\frac{1}{2}})_{j,j} + \sum_{i,j=1}^{m-1} v_i v_j \left( \mathcal{G}(X,d^{\frac{1}{2}})_{i,i} + \mathcal{G}(X,d^{\frac{1}{2}})_{j,j} - 2 \mathcal{G}(X,d^{\frac{1}{2}})_{i,j} \right) \\
        &= 2 \left( \sum_{i=1}^m \tilde{v}_i \right) \left( \sum_{j=1}^{m-1} v_j \mathcal{G}(X,d^{\frac{1}{2}})_{j,j} \right) - 2 \sum_{i,j=1}^{m-1} v_i v_j \mathcal{G}(X,d)_{i,j} \\
        &= - 2 \sum_{i,j=1}^{m-1} v_i v_j \mathcal{G}(X,d^{\frac{1}{2}})_{i,j}.
    \end{align*}
    $(X,d)$ is of negative type so $\mathcal{G}(X,d^{\frac{1}{2}})$ is positive semi-definite, thus according to Corollary \ref{embed} $(X,d^{\frac{1}{2}})$ is Euclidean.
\end{itemize}

\end{proof}

\begin{remark}\label{rem}
    If $(X,d)$ is a metric space, $(X,d^2)$ may not be a metric space, but the proof still shows that $(X,d)$ is Euclidean if and only if for all $\eta \in \1^{\perp}, \sum_{i,j=1}^m \eta_i \eta_j d(x_i,x_j)^2 \leq 0$.
\end{remark}

\subsection{Further properties of power transform metric spaces and Euclidean metric spaces}

We use together the results of the last three subsections in order to see how these notions combine with each other. These results are discussed in the following articles \cite{bogomolny2007distance, maron2018probably, schoenberg1937, Schoenberg1938}.

With Remark \ref{rem} we add another characterization of Euclidean metric spaces to Corollary \ref{embed}, using some computational tricks.

\begin{theorem}
The following statements are equivalent :
\begin{enumerate}
    \item $(X,d)$ is Euclidean
    \item $\mathcal{G}(X,d)$ is positive semi-definite
    \item $\forall \eta \in \1^{\perp}, \sum_{i,j=1}^n \eta_i \eta_j d(x_i,x_j)^2 \leq 0$
    \item $\forall \eta \in \1^{\perp}, \forall \lambda > 0, \sum_{i,j=1}^n e^{-\lambda d(x_i,x_j)^2} \eta_i \eta_j \geq 0 $.
\end{enumerate}
\end{theorem}

\begin{proof}
We already proved $(1) \iff (2) \iff (3)$.

If $(X,d)$ is Euclidean, suppose $X \subset \R^m$ for some $m \in \N$.

By inverse Fourier transform of a Gaussian , for $z \in \R^m$ and $\lambda>0$, we have : $$e^{-\lambda^2 ||z||_2^2} = \frac{1}{(4\pi)^{\frac{m}{2}}} \int_{\R^m} e^{i <\xi,\lambda z>} e^{- ||\xi||_2^2} d\xi. $$

So for $\eta \in \1^{\perp}$,
\begin{align*}
    \sum_{j,k=1}^n \eta_j \eta_k e^{- \lambda^2 \|x_j-x_k \|^2} &= \frac{1}{(4\pi)^{\frac{m}{2}}} \int_{\R^m} \sum_{j,k=1}^n e^{i <\xi, \lambda (x_j-x_k)>} e^{-||\xi||_2^2} d\xi \\
    &= \frac{1}{(4\pi)^{\frac{m}{2}}} \int_{\R^m} \sum_{j,k=1}^n e^{i \lambda <\xi, x_j>} \overline{e^{i \lambda <\xi, x_k>}} e^{-||\xi||_2^2} d\xi \\
    &= \frac{1}{(4\pi)^{\frac{m}{2}}} \int_{\R^m}  \left| \sum_{j=1}^n e^{i \lambda <\xi, x_j>} \right|^2 e^{-||\xi||_2^2} d\xi \\
    &\geq 0.
\end{align*}
Replace $\lambda^2$ by $\lambda$ to obtain (4).

Reciprocally suppose we have $(4)$.

Let $0 < \gamma < 2$, and  $I = \int_0^{+\infty} (1-e^{-x^2}) x^{-1-\gamma} dx$ (well-defined).

For $t>0$ by change of variables we have $t^{\gamma} = \frac{1}{I} \int_0^{+\infty} (1-e^{-t^2x^2}) x^{-1-\gamma} dx $.
Then we have for $\eta \in \1^{\perp}$:
\begin{align*}
    \sum_{i,j=1}^n \eta_i \eta_j d(x_i,x_j)^\gamma &= \sum_{i,j=1}^n \eta_i \eta_j \frac{1}{I} \int_0^{+\infty} (1-e^{-d(x_i,x_j)^2x^2}) x^{-1-\gamma} dx \\
    &= - \frac{1}{I}   \int_0^{+\infty} \left( \sum_{i,j=1}^n \eta_i \eta_j e^{-d(x_i,x_j)^2x^2} \right) x^{-1-\gamma} dx \\
    &\leq 0.
\end{align*}

It is true for all $0 < \gamma < 2$ so by continuity of the left hand side with respect to $\gamma$ it also true for $\gamma = 2$ so
$$\sum_{i,j=1}^n \eta_i \eta_j d(x_i,x_j)^2 \leq 0 .$$
So we have $(2)$.

\end{proof}

With the same trick we prove that being Euclidean is preserved through power transform.

\begin{theorem}\label{power_euclidean}
    If $(X,d)$ is an Euclidean metric space then for $0<c \leq 1$, $(X,d^c)$ is Euclidean
\end{theorem}

\begin{proof}
With the same notations, for $\eta \in \1^{\perp}$
$$\sum_{i,j=1}^n \eta_i \eta_j d(x_i,x_j)^{2c}= - \frac{1}{I}   \int_0^{+\infty} \left( \sum_{i,j=1}^n \eta_i \eta_j e^{-d(x_i,x_j)^2x^2} \right) x^{-1-2c} dx \leq 0$$ by the last proposition because $(X,d)$ is Euclidean. Thus by the same proposition, $(X,d^c)$ is Euclidean.
\end{proof}

We have thus an important family of spaces that have the property of being of negative type.

\begin{corollary}\label{eucl}
Euclidean metric spaces are of negative type.
\end{corollary}

\begin{proof}
If $(X,d)$ is Euclidean, by last theorem $(X,d^{\frac{1}{2}})$ is also Euclidean so by Remark \ref{corres} it is of negative type.
\end{proof}

\subsection{Proof that spaces with 4 points or less points are of negative type}

The last corollary allows us to conclude for 2 and 3 points metric spaces. Indeed, they are always embedabble in $\R^2$ so they are of negative type. We prove that the $n=4$ case is also true.

\begin{theorem}\label{four}
If $(X,d)$ is a metric space with $|X|=4$ then $(X,d)$ is of negative type.
\end{theorem}

\begin{proof}
If $(X,d)$ is of negative type then $(X,d^{\frac{1}{2}})$ is Euclidean so by Theorem \ref{power_euclidean} for all $0 \leq \alpha \leq \frac{1}{2}, (X,d^{\alpha})$ is Euclidean.

By contradiction, suppose that there exists a metric space $(X,d)$ with $|X|=4$ and $0 \leq \alpha \leq \frac{1}{2}$ such that $(X,d^{\alpha})$ is not Euclidean.

Thus by Corollary \ref{embed} $\mathcal{G}(X,d^{\alpha})$ is not positive. So by contraposition of Sylvester's criterion for positive semi-definite matrices, there exists a principal minor of $\mathcal{G}(X,d^{\alpha})$ that is negative, i.e. , there exists $I \subset \llbracket 1,3 \rrbracket$ such that $\det(\mathcal{G}(X,d^{\alpha})_I) < 0$. 
$\mathcal{G}(X,d^{\alpha})_I = \mathcal{G}( \lbrace p_0, p_i, i \in I \rbrace ,d^{\alpha}) $, however 2 and 3 points metric spaces are Euclidean so necessarily, $I = \llbracket 1,3 \rrbracket$. Finally, $\det(\mathcal{G}(X,d^{\alpha})) < 0$.

Let for $0 \leq t \leq \frac{1}{2}, g(t) = \det \left( \mathcal{G}(X,d^t) \right)$. $g$ is continuous. We have $g(\alpha) < 0$. Also, $g(0) > 0$ as $(X,d^0)$ (with $0^0=0$ here) has the same matrix distance as a tetrahedron. So there exists $\beta \in \left] 0, \alpha \right[$ such that $g(\beta) = 0$.

Thus, $\mathcal{G}(X,d^{\beta})$ is positive semi-definite and of rank less than 2 so $(X,d^{\beta})$ is embedabble in $\R^2$ by Theorem \ref{embed0}. It follows that there exists $x,y,z \in X$ distincts such that $\angle xyz \geq \frac{\pi}{2}$. Let's denote $p,q,r$ the distances between these three points in $(X,d)$. By triangular inequality $p+q > r$ so by {Lemma \ref{ineq_power}} because $2 \beta < 1$ we have $(p^{\beta})^2 + (q^{\beta})^2 > (r^{\beta})^2$. This is true for every permutation between $p, q, r$ so $(\lbrace x,y,z \rbrace, d^{\beta})$ is an acute triangle.

Contradiction so $(X,d)$ is of negative type.

\end{proof}

\subsection{The case of spaces with more than 5 points}

This proof does not hold for $n \geq 5$ as the following example yields a $5$-point metric space that is not of negative type.

\begin{example}
We introduce the following graph with the shortest path distance :
\begin{figure}[h]
    \centering
    \includegraphics[scale=0.1]{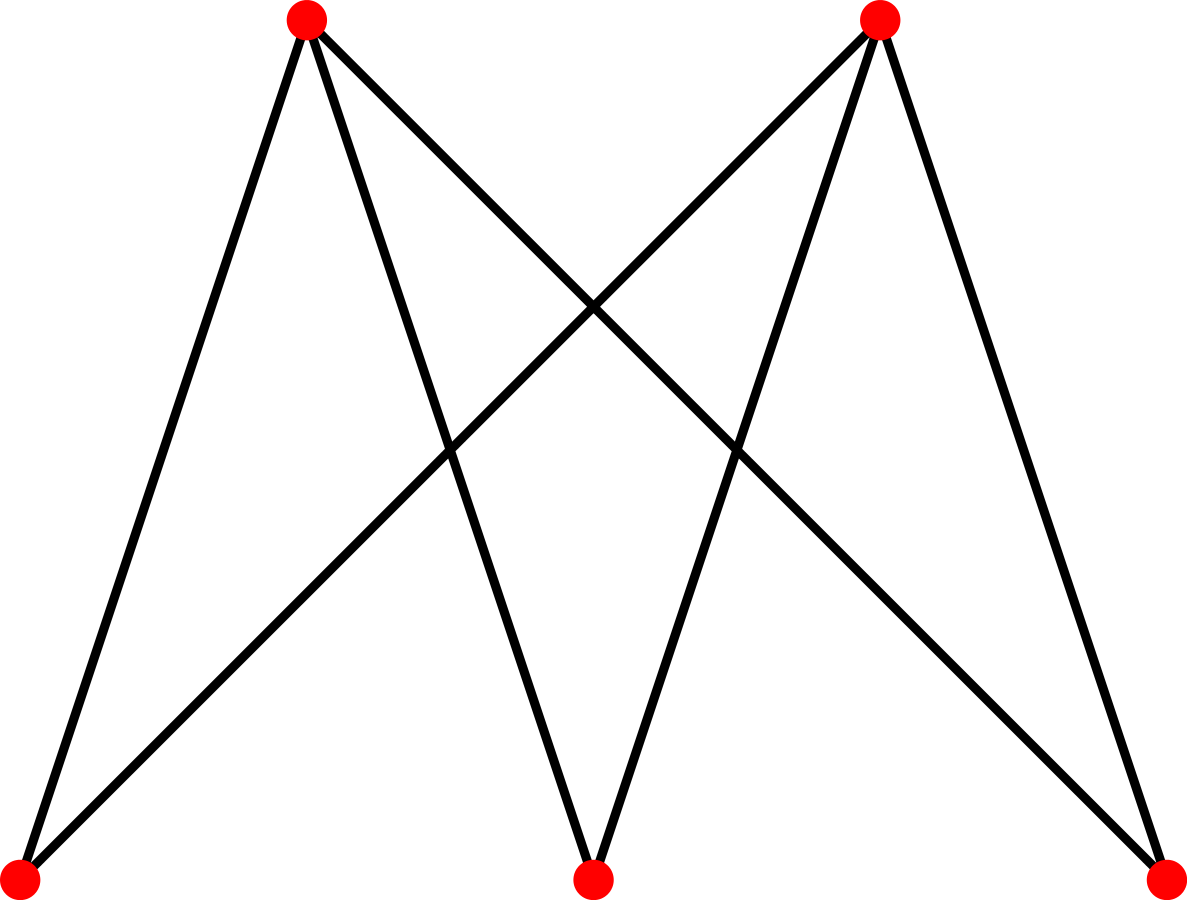}
    \caption{Complete bipartite graph $K_{3,2}$}
    \label{fig:my_label}
\end{figure}

Its distance matrix is :
$$ D = 
\begin{pmatrix}
     0 &2 &2 &1 &1 \\
     2 &0 &2 &1 &1 \\
     2 &2 &0 &1 &1 \\
     1 &1 &1 &0 &2 \\
     1 &1 &1 &2 &0
\end{pmatrix}.
$$

Let $\eta = (1 \quad 1 \quad 1 \quad -1 \quad -2)^T$, we have $\eta \in \1^{\perp}$ and $\eta^TD\eta = 2 >0$.

\end{example}

There exists metric spaces that are not negative type for all $n \geq 5$ according to \cite{hjorth1998finite} Prop. 3.7 :

\begin{proposition}
Let $r>4$, let $M_r$ be the metric space with $r+2$ points with the following distance :
$d(p_i, p_j) = 2 \1_{i \neq j}$ if $i \leq r$ and $j \leq r$, and $d(p_i, p_j) = \1_{i \neq j}$ otherwise. Then $(M_r, d)$ is a metric space that is not of negative type.
\end{proposition}

However, there is no known counter example to equality (\ref{problem}) at this point.

\end{document}